\documentclass[12pt]{amsart}
\usepackage{epsfig} 
\usepackage{color}
\usepackage{psfrag}
\usepackage{graphicx}
\usepackage{amssymb}
\usepackage{amsmath}
\usepackage{latexsym}
\usepackage{amssymb,latexsym}
\usepackage{enumerate}
\usepackage[mathscr]{eucal}
\usepackage[isolatin]{inputenc}
\usepackage{upref}
\usepackage{float}
\usepackage[left=2.5cm,right=2.5cm,top=2.5cm,bottom=2.5cm,includeheadfoot]{geometry}
\makeatletter
\@namedef{subjclassname@2010}{%
 \textup{2010} Mathematics Subject Classification}
\makeatother

\newtheorem{thm}{Theorem}[section]

\newtheorem{lem}[thm]{Lemma}

\theoremstyle{definition}
\newtheorem{defin}[thm]{Definition}
\newtheorem{rem}[thm]{Remark}

\numberwithin{equation}{section}

 \newcommand{\setC}{\mathbb{C}}
 
 \newcommand{\setN}{\mathbb{N}}
 
 \newcommand{\setQ}{\mathbb{Q}}
 \newcommand{\setR}{\mathbb{R}}
 
 \newcommand{\setZ}{\mathbb{Z}}

\begin{document}

\baselineskip=17pt

\title[]
{New insight into results of Ostrowski and Lang\\
on sums of remainders using Farey sequences}

\author{Matthias Kunik}
\address{Universit\"{a}t Magdeburg\\
IAN \\
Geb\"{a}ude 02 \\
Universit\"{a}tsplatz 2 \\
D-39106 Magdeburg \\
Germany}
\email{matthias.kunik@ovgu.de}

\date{\today}

\begin{abstract}
The sums $S(x,t)$ of the centered remainders $kt-\lfloor kt\rfloor - 1/2$ over $k \leq x$
and corres\-ponding Dirichlet series
were studied by A. Ostrowski, E. Hecke, H. Behnke and S. Lang
for fixed real irrational numbers $t$. Their work was originally inspired by Weyl's equidistribution results modulo 1 for sequences in number theory.

In a series of former papers we obtained limit functions which describe scaling pro\-per\-ties of the Farey sequence of order $n$ for $n \to \infty$  in the vicinity of any fixed
fraction $a/b$ and which are independent of $a/b$. We extend this theory on the sums $S(x,t)$
and also obtain a scaling behaviour with a new limit function. 
This method leads to a refinement of results given by Ostrowski and Lang and establishes
a new proof for the analytic continuation of related Dirichlet series.
We will also present explicit relations to the theory of Farey sequences.
\end{abstract}



\maketitle

\section{Introduction} \label{farey_intro}

In \cite{Weyl1} Hermann Weyl developed a general and far-reaching theory for the equidistribution of sequences modulo 1, which is discussed from a historical point of view in Stammbach's paper \cite{Stammbach}. Especially Weyl's result that for real $t$ the sequence 
$
t ,\, 2t ,\, 3t,\,\ldots
$
is equidistributed modulo 1 if and only if $t$ is irrational
can be found in \cite[\S 1]{Weyl1}. This means that
\begin{equation*}
\lim \limits_{N \to \infty} \frac{1}{N} 
\#\left\{nt -\lfloor nt \rfloor \in [a,b] : n \leq N \right\}=b-a
\end{equation*}
holds for all subintervals $[a,b] \subseteq [0,1]$
if and only if $t \in \setR \setminus \setQ$.
Here $\#M$ denotes the number of elements of a finite set $M$.
This generalization of Kronecker's Theorem \cite[Chapter XXIII, Theorem 438]{hw} is an important result in number theory. We have only mentioned its one dimensional version, but the higher dimensional case is also treated in Weyl's paper.

Now we put 
\begin{equation}\label{Sxdef}
S(x,t)=\sum \limits_{k \leq x}
\left(kt -\lfloor kt \rfloor-\frac12\right)
\end{equation} 
for $x \geq 0$ and $t \in \setR$.
If the sequence $(n t)_{n \in \setN}$ is 
``well distributed" modulo 1 for irrational $t$, then $|S(x,t)|/x$ should be ``small"
for $x$ large enough. 

In \cite[Equation (2), p. 80]{Ost1922} Ostrowski used the continued fraction expansion $t=\langle \lambda_0,\lambda_1,\lambda_2,\ldots \rangle$ for irrational $t$ and presented a very efficient calculation of $S(n,t)$ with $n \in \setN_0$. He namely obtained a simple iterative procedure using
at most $\mathcal{O}(\log n)$ steps for $n \to \infty$, uniformly in $t \in \setR$. We have summarized his result in Theorem \ref{ostrowski} of the paper on hand.
From this theorem he derived an estimate for $S(n,t)$ in the case of irrational 
$t \in \setR$ which depends on the choice of $t$.
Especially if $(\lambda_k)_{k \in \setN_0}$ is a bounded sequence, then we say that $t$ has \textit{bounded partial quotients}, and have in this case from Ostrowski's paper 
\begin{equation}\label{ostopt}
|S(n,t)| \leq C(t) \log n \,,\quad n \geq 2\,,
\end{equation} 
with a constant $C(t)>0$ depending on $t$\,. Ostrowski also showed that this 
gives the best possible result, answering an open question posed by Hardy and Littlewood.\\

In \cite{Lang1966} and \cite[III,\S 1]{Lang1995} Lang obtained
for every fixed $\varepsilon > 0$ that 
\begin{equation}\label{lang_fast_ueberall}
|S(n,t)| \leq\left(\log n \right)^{2+\varepsilon} \quad
\mbox{~for ~} n \geq n_0(t,\varepsilon)
\end{equation}
for almost all $t \in \setR$ with a constant $n_0(t,\varepsilon) \in \setN$.
Let $\alpha$ be an irrational real number and $g \geq 1$ be an increasing function,
defined for sufficiently large positive numbers. Due to Lang
\cite[II,\S 1]{Lang1995} the number $\alpha$ is of type $\leq g$ if for all sufficiently large numbers $B$, there exists a solution in relatively prime integers $q,p$ of the inequalities
$$
|q\alpha-p|<1/q\,, \quad B/g(B)\leq q <B\,.
$$
After Corollary 2 in \cite[II,\S 3]{Lang1995}, where Lang studied the quantitative connection between Weyl's equidistribution modulo 1 for the sequence $t,2t,3t,\ldots$ and the type of the irrational number $t$, he mentioned the work of Ostrowski \cite{Ost1922} and Behnke \cite{Behnke} and wrote:
"Instead of working with the type as we have defined it, however,
these last-mentioned authors worked with a less efficient way of determining the approximation behaviour of $\alpha$ with respect to $p/q$, whence followed weaker results and more complicated proofs." \\

Though Lang's theory gives Ostrowski's estimate \eqref{ostopt}
for all real irrational numbers $t$ with bounded partial quotients, 
see \cite[II, \S 2, Theorem 6 and III,\S 1, Theorem 1]{Lang1995}, as well as
estimate \eqref{lang_fast_ueberall} for almost all $t \in \setR$, Lang did not use
Ostrowski's efficient formula for the calculation of $S(n,t)$\,.
We will see in Section \ref{dirichlet_section} of the paper on hand that Ostrowski's formula can be used as well in order to derive estimate \eqref{lang_fast_ueberall} for almost all $t \in \setR$, without working with the type defined in \cite[II,\S 1]{Lang1995}. For this purpose we will present the general and useful Theorem \ref{kettenabschaetz}, which will be derived in Section \ref{sawtooth} from the elementary theory of continued fractions. Our resulting new Theorems \ref{B0_mass_almost_everywhere},
\ref{B0_mass} now have the advantage to provide an explicit form for those sets of $t$-values
which satisfy crucial estimates of $S(n,t)$.\\
If $\Theta:[1,\infty) \to [1,\infty)$ is any monotonically increasing function
with $\begin{displaystyle}\lim \limits_{n \to \infty} \Theta(n)=\infty\end{displaystyle}$\,,
then Theorem \ref{B0_mass} gives the inequality
$
|S(n,t)| \leq 2 \log^2(n) \Theta(n)
$
uniformly for all $n \geq 3$ and all $t \in \mathcal{M}_n$ for a sequence of sets
$\mathcal{M}_n \subseteq [0,1]$ with $\lim \limits_{n \to \infty}|\mathcal{M}_n|=1$.
Here $|\mathcal{M}_n|$ denotes the Lebesgue-measure of $\mathcal{M}_n$.
On the other hand Theorem \ref{Bx_L2} states that
$$
\left(\int \limits_{0}^{1}S(n,t)^2\,dt\right)^{1/2}=\mathcal{O}(\sqrt{n}) 
\quad \text{for~}n \to \infty
$$
gives the true order of magnitude for the $L_2(0,1)$-norm of $S(n,\cdot)$.
If $\Theta$ increases slowly then the values of $S(n,t)$ with $t$ in the unit-interval $[0,1]$ 
which give the major contribution to the $L_2(0,1)$-norm have their pre-images only in the small complements $[0,1] \setminus \mathcal{M}_n$. We see that $n_0(t,\varepsilon)$ 
in estimate \eqref{lang_fast_ueberall} 
depends substantially on the choice of $t$.
Moreover, a new representation formula for $B_n(t)=S(n,t)/n$
given in Section \ref{sawtooth}, Theorem \ref{Bx_thm} will also give an alternative proof of Ostrowski's estimate \eqref{ostopt} if $t$ has bounded partial quotients.
In this way we summarize and refine the corresponding results given by Ostrowski and Lang,
respectively.\\

For $n \in \setN$ and $N = \sum\limits_{k=1}^{n}\varphi(k)$
with Euler's totient function $\varphi$ 
the Farey sequence $\mathcal{F}_n$ of order $n$ consists of all reduced and ordered fractions
\begin{equation*}
\frac{0}{1}=\frac{a_{0,n}}{b_{0,n}} < \frac{a_{1,n}}{b_{1,n}} < \frac{a_{2,n}}{b_{2,n}} 
< \ldots < \frac{a_{N,n}}{b_{N,n}}=\frac{1}{1}
\end{equation*}
with $1 \leq b_{\alpha,n} \leq n$ for $\alpha = 0,1, \ldots,N$.
By $\mathcal{F}^{ext}_n$ we denote the extension of $\mathcal{F}_n$
consisting of all reduced and ordered fractions $\frac{a}{b}$ with $a \in \setZ$ 
and $b \in \setN$, $b \leq n$.

In the former paper \cite{Ku4} we have studied 1-periodic functions 
$\Phi_n : \setR \to \setR$ which are related to the Farey sequence $\mathcal{F}_n$,
based on the theory developed in \cite{Ku1, Ku2, Ku3} for related functions.
For $k \in \setN$ and $x>0$ we use the M\"obius function $\mu$ and define
the 1-periodic functions $q_k, \Phi_x : \setR \to \setR$ given by
\begin{equation}\label{familien}
\begin{split}
q_k(t) &=-\sum \limits_{d | k} \mu(d)\,
\beta \left(\frac{kt}{d}\right) ~\,
\mbox{with}~\, \beta(t)=t-\lfloor t \rfloor - \frac12\,,\\
\Phi_x\left(t \right) &= \frac{1}{x} \sum \limits_{k \leq x} 
q_k(t)=- \frac{1}{x} \sum \limits_{j \leq x}
\sum \limits_{k \leq x/j} \mu(k) \beta\left(jt \right)\,.\\
\end{split}
\end{equation}
The functions $\Phi_x$ determine the number of Farey fractions in prescribed intervals.
More precisely, $t\sum\limits_{k \leq n}\varphi(k)+n\Phi_n(t)+\frac12$ gives the number
of fractions of $\mathcal{F}^{ext}_n$ in the interval $[0,t]$ 
for $t \geq 0$ and $n \in \setN$. Moreover, there is a connection between the functions $S(x,t)$ and $\Phi_x(t)$ via the Mellin-transform and the Riemann-zeta function, namely the relation
\begin{equation*}
\int \limits_{1}^{\infty} \frac{S(x,t)}{x^{s+1}}dx
=-\zeta(s)\,\int \limits_{1}^{\infty} \frac{\Phi_x(t)}{x^s}dx\,,
\end{equation*}
valid for $\Re(s)>1$ and any fixed $t \in \setR$.
We will use it in a modified form in Theorem \ref{F_thm}.

In contrast to Ostrowski's approach using elementary evaluations of $S(n,t)$ 
for real values of $t$, Hecke \cite{Hecke} considered the case of special quadratic irrational numbers $t$, studied the analytical properties of the corresponding Dirichlet series
\begin{equation}\label{Hecke_series}
\sum \limits_{m=1}^{\infty}\frac{mt-\lfloor mt \rfloor-\frac12}{m^s}=
s\int \limits_{1}^{\infty} \frac{S(x,t)}{x^{s+1}}dx
\end{equation}
and obtained its meromorphic continuation to the whole complex plane, 
including the location of poles. Hecke could use his analytical method to derive estimates 
for $S(n,t)$, but he did not obtain
Ostrowski's optimal result \eqref{ostopt}
for real irrationalities $t$ with bounded partial quotients.

For positive irrational numbers $t$ 
Sourmelidis \cite{Sourmelidis} studied analytical relations 
between the Dirichlet series in \eqref{Hecke_series} and the so called Beatty zeta-functions and Sturmian Dirichlet series.

For $x>0$ we set
\begin{equation*}
r_x =\frac{1}{x} \sum \limits_{k \leq x} \varphi(k)\,,
\quad
s_x = \sum \limits_{k \leq x} \frac{\varphi(k)}{k} \,,
\end{equation*}
and define the \textit{continuous} and odd function $h : \setR \to \setR$ by
\begin{equation*}
h(x)=\begin{cases}
0 &\ \text{for}\ x =0 \,,\\
3x/\pi^2+r_x-s_x &\ \text{for}\ x >0 \,,\\
-h(-x) &\ \text{for}\ x < 0\, .\\
\end{cases}
\end{equation*}
Then we obtained in \cite[Theorem 2.2]{Ku4} 
for any fixed reduced fraction $a/b$ with $a \in \setZ$ and $b \in \setN$
and any $x_*>0$ that for $n \to \infty$
$$
\tilde{h}_{a,b}(n,x)=-b \, \Phi_n\left(\frac{a}{b}+\frac{x}{bn}\right) 
$$
converges uniformly to $h(x)$ for $-x_* \leq x \leq x_*$.
For this reason we have called $h$ a limit function.
It follows from \cite[Theorem 1]{Mont} with an absolute constant $c>0$ 
for $x\geq 2$ that
$\begin{displaystyle}
h(x) = \mathcal{O}\left( e^{-c\sqrt{\log x}}\right)\,.
\end{displaystyle}$
Plots of this limit function are presented in
Section \ref{appendix}, Figures \ref{h25eps},\ref{h50eps},\ref{h500eps}.

In Section \ref{sawtooth} we introduce another limit function ${\tilde \eta} : \setR \to \setR$ defined by ${\tilde \eta}(0)=-\frac12$ and
\begin{equation*}
{\tilde \eta}(x)=
\frac{(x-\lfloor x \rfloor)(x-\lfloor x \rfloor-1)}{2x}
 \quad \text{for}\ x \in \setR \setminus \{0\}\,,
\end{equation*}
and obtain from Theorem \ref{rescaled_limit} for $B_n(t)=S(n,t)/n$
analogous to \cite[Theorem 2.2]{Ku4} the new result that for $n \to \infty$
$$
\tilde{\eta}_{a,b}(n,x)= b \, B_n\left(\frac{a}{b}+\frac{x}{bn}\right) 
$$
converges uniformly to $\tilde{\eta}(x)$ for $-x_* \leq x \leq x_*$.
A plot of  ${\tilde \eta}(x)$ for $-8 \leq x \leq 8$ is given in
Section \ref{appendix}, Figure \ref{eta8eps}.
Now Theorem \ref{Bx_thm}(b) follows from part (a) and leads to
the formula \eqref{Bsequence}, which bears a strong resemblance 
to that in Ostrowski's Theorem \ref{ostrowski} and gives an alternative proof for Ostrowski's estimate \eqref{ostopt} if $t$ has bounded partial quotients.
Hence it would be interesting to know whether there is a deeper reason for this analogy.

\section{Sums with sawtooth functions}\label{sawtooth}

With the sawtooth function $\beta(t)=t-\lfloor t \rfloor -\frac12$ we define
for $x>0$ the 1-periodic
functions $B_{x}: \setR \to \setR$ by
\begin{equation}\label{Bxdef}
B_{x}\left(t \right)= \frac{1}{x}\sum \limits_{k \leq x} \beta(kt)\,.
\end{equation}

Next we will state \cite[Theorem 2.2]{Ku3} which,
amongst other things, connects the study of the functions $B_x$
with the theory of Farey fractions. 

\begin{thm}\label{farey_thm2} \cite[Theorem 2.2]{Ku3}
Assume that $\frac{a}{b} < \frac{a^*}{b^*}$
are consecutive reduced fractions in the extended Farey sequence $\mathcal{F}^{ext}_b$ 
of order $b \leq n$ with $b,b^*,n \in \setN$. For $q \geq 0$ we define
\begin{equation}\label{xip}
\xi_+(q) = \frac{a^*+aq}{b^*+bq}\,, \qquad x_+(q) = \frac{n}{b^*+bq}\,,
\end{equation}
and see that its inverse functions
\begin{equation*}
\xi_+^{-1}(\xi) = \frac{a^*-b^*\xi}{b\xi -a}\,, \qquad x_+^{-1}(x) = \frac{n/x-b^*}{b}\,
\end{equation*}
are defined for $a/b < \xi \leq a^*/b^*$ and $0 < x \leq n/b^*$, respectively.
\begin{itemize}
\item[(a)]
We assume that 
\begin{equation*}
\frac{a}{b} < \frac{A}{B} \leq \frac{a^*}{b^*}
\end{equation*}
with the reduced fraction $\xi=\frac{A}{B} \in \mathcal{F}^{ext}_n$, $A \in \setZ$, $B \in \setN$,
and put
\begin{equation*}
q = \frac{Ba^*-Ab^*}{Ab-Ba} \,, \qquad
\alpha = \left \lfloor  x_+(q)  \right \rfloor\,.
\end{equation*}
Then $\alpha, Ab-Ba \in \setN$, and $q$ 
is reduced with 
$
q = \xi_+^{-1}(\xi) \in \mathcal{F}^{ext}_{\alpha}\,.
$
\item[(b)]
Let $0 \leq q=\frac{a'}{b'}$ be reduced,
assume that $\alpha = \left \lfloor  x_+(q)  \right \rfloor \geq 1$
and that $q \in \mathcal{F}^{ext}_{\alpha}$. We put
\begin{equation*}
\xi=\frac{a^*b'+aa'}{b^*b'+ba'}\,.
\end{equation*}
Then $\xi$ is a reduced fraction of $\mathcal{F}^{ext}_n$
in the interval $(a/b,a^*/b^*]$ satisfying
$\xi=\xi_+(q)$\,.
\end{itemize}
\end{thm}

The function $\beta(t)$ has jumps of height $-1$ exactly at integer numbers $t \in \setZ$
but is continuous elsewhere.
Let $a/b$ with $a \in \setZ$, $b \in \setN$ be any reduced fraction with denominator $b \leq x$.

By $\begin{displaystyle} u^{\pm}(t)=\lim \limits_{\varepsilon \downarrow 0} 
u(t \pm \varepsilon) \end{displaystyle}$
we denote the one-handed limits of a real- or complex valued function $u$ with respect to the real variable $t$.

Then the height of the jump of $B_x$ at $a/b$ is given by
\begin{equation}\label{Bxjump}
B_x^+(a/b)-B_x^-(a/b)=-\frac{1}{x}\left\lfloor\frac{x}{b}\right\rfloor\,.
\end{equation}
We introduce the function ${\tilde \eta} : \setR \to \setR$ given by ${\tilde \eta}(0)=-\frac12$ and
\begin{equation}\label{etalimit}
{\tilde \eta}(x)=
\frac{(x-\lfloor x \rfloor)(x-\lfloor x \rfloor-1)}{2x}
 \quad \text{for}\ x \in \setR \setminus \{0\}\,.
\end{equation}
The function ${\tilde \eta}$ is continuous apart from the zero-point with derivative 
\begin{equation}\label{etader}
{\tilde \eta}'(x)= \frac12-\frac{\lfloor x \rfloor (\lfloor x \rfloor+1)}{2x^2}\quad \text{for} ~  
x \in \setR \setminus \setZ\,.
\end{equation}

In the following theorem we assume that $\frac{a}{b} < \frac{a^*}{b^*}$
are conse\-cu\-tive reduced fractions in the extended Farey sequence $\mathcal{F}^{ext}_b$ 
of order $b \leq n$ with $b,b^*,n \in \setN$.

\begin{thm}\label{Bx_thm}
\begin{itemize}
\item[(a)]
For $0<x\leq n/b^*$ we have
\begin{equation*}
\begin{split}
&B_{n}\left(\frac{a}{b}+\frac{x}{bn}\right)
	=B_{n}\left(\frac{a}{b}\right)+\frac{1}{2b}+\frac{{\tilde \eta}(x)}{b}
	-\frac{x}{n}\, B_{x}^-\left(\frac{n/x-b^*}{b}\right)\\
&+\frac{x}{2bn} +\frac{1}{n}\,\sum \limits_{k \leq x}\beta\left(\frac{n-kb^*}{b}\right)\,.\\
\end{split}
\end{equation*}
\item[(b)] For $0<t \leq 1$ and $n \in \setN$ we have
\begin{equation*}
B_{n}\left(t\right)
	={\tilde \eta}(tn)
	-\frac{\lfloor tn \rfloor}{n}\, B_{\lfloor tn \rfloor}^-\left(\frac{1}{t}-
	\left \lfloor \frac{1}{t} \right \rfloor\right)
+\frac{tn-\lfloor tn \rfloor}{2n}\,.
\end{equation*}
\end{itemize}
\end{thm}

\begin{proof} Since (b) follows from (a) in the special case $a=0$, $a^*=b^*=b=1$,
it is sufficient to prove (a). We define for $0 < x \leq n/b^*$:
	\begin{equation}\label{rntilde}
	\begin{split}
	R_n\left(\frac{a}{b},x\right) &=-b\left( 
	B_{n}\left(\frac{a}{b}+\frac{x}{bn}\right)-
	B_{n}\left(\frac{a}{b}\right)\right)\\
	&+\frac12+{\tilde \eta}(x)-\frac{bx}{n}\, B_{x}^-\left(\frac{n/x-b^*}{b}\right)+\frac{x}{2n}\,.\\
		\end{split}
	\end{equation}

	We use \eqref{Bxdef}, \eqref{etader} and obtain, except of the discrete set of jump discontinuities 
	of $R_n$, its derivative
	\begin{equation*}
	\begin{split}
	 \frac{d}{dx} R_n\left(\frac{a}{b},x\right) &=-b\cdot\frac{n+1}{2}\cdot \frac{1}{bn} 
		+\frac{1}{2}-\frac{\lfloor x \rfloor (\lfloor x \rfloor+1)}{2x^2}\\
		&-\frac{b}{n}\,\frac{d}{dx}
		\left(
		 x \, B_{x}^-\left(\frac{n/x-b^*}{b}\right)
		\right)+\frac{1}{2n}\\
		& = -\frac{\lfloor x \rfloor (\lfloor x \rfloor+1)}{2x^2}-\frac{b}{n}\,\frac{d}{dx}
		\sum \limits_{k \leq \lfloor x \rfloor}
		\beta^-\left(k \, \frac{n/x-b^*}{b}\right)\\
		& = -\frac{\lfloor x \rfloor (\lfloor x \rfloor+1)}{2x^2}-\frac{b}{n}
		\sum \limits_{k \leq \lfloor x \rfloor}k
		\cdot \frac{n}{b} \cdot \left(-\frac{1}{x^2}\right)=0\,.\\
		\end{split}
	\end{equation*}
	Note that $B_n=B_n^+$ and $R_n=R^+_{n}$. 
	We deduce from Theorem \ref{farey_thm2} for any $x$ in the interval
	$0 < x \leq n/b^*$  that 
	$\begin{displaystyle} a/b+x/(bn)\end{displaystyle}$
	is a jump discontinuity of $B_{n}$ if and only if
	$\begin{displaystyle}(n/x-b^*)/b\end{displaystyle}$
	is a jump discontinuity of $B_{x}$. 
	Let $x_+(q)$ be defined by the second equation in \eqref{xip} and let 
	$$
	q =\frac{a'}{b'}=\frac{n/x_+(q)-b^*}{b}
	$$
	be any reduced fraction $a'/b' \in  \mathcal{F}^{ext}_{\lfloor x_+(q) \rfloor}$ 
	from Theorem \ref{farey_thm2}(b). We use \eqref{Bxjump} and have
	\begin{equation}\label{phin_sprung}
    -b (B_{n}^{+}-B_{n}^{-})\left(\frac{a}{b}+\frac{x_+(q)}{nb}\right) 
		=	\frac{b}{n} \left\lfloor \frac{n}{b^*b'+ba'} \right\rfloor
		=	\frac{b}{n} \left\lfloor \frac{x_+(q)}{b'} \right\rfloor\,.		
	\end{equation}
	
	First we consider the case that $x_+(q)$ is a {\it non-integer} number. 
	Using again \eqref{Bxjump}
	we obtain 
	\begin{equation}\label{bxsprung}
	\begin{split}
   & \lim \limits_{x \,\downarrow \,x_+(q)}	
     \left(x \,B_{x}^-\left(\frac{n/x-b^*}{b}\right)
		\right)-
		    \lim \limits_{x \,\uparrow \, x_+(q)}	
     \left(x \,B_{x}^-\left(\frac{n/x-b^*}{b}\right)
		\right)\\
		&=x_+(q)\cdot\left(B_{x_+(q)}^- -B_{x_+(q)}^+ \right)\left(\frac{a'}{b'}\right)
		=\left\lfloor \frac{x_+(q)}{b'} \right\rfloor\,,
		\end{split}
  \end{equation}
 taking into account that $(n/x-b^*)/b$ is monotonically decreasing with respect to $x$.
For \eqref{phin_sprung} and \eqref{bxsprung} we note that $a^*b-ab^*=1$ 
for the Farey fractions $a/b < a^*/b^*$ in Theorem \ref{farey_thm2} 
and recall that $B_n$ has a jump at
$$
\frac{a}{b}+\frac{x_+(q)}{nb}=\frac{a(b^*+bq)+a^*b-ab^*}{b(b^*+bq)}=
\frac{a^*b'+aa'}{b^*b'+ba'}=\xi_+(q)
$$
if and only if $B_{x_+(q)}$ has a jump at $q=a'/b'$.
We obtain from \eqref{rntilde}, \eqref{phin_sprung}, \eqref{bxsprung} that
	\begin{equation*}
     \left(R_n^{+}-R_n^{-}\right)\left(\frac{a}{b},x_+(q)\right) 
		=	\frac{b}{n} \left\lfloor \frac{x_+(q)}{b'} \right\rfloor
		-\frac{b}{n} \left\lfloor \frac{x_+(q)}{b'} \right\rfloor=0\,.		
	\end{equation*}
This implies that $R_n$ is free from jumps at non-integer arguments $x$.
 It remains to calculate the jumps of $R_n$ at any {\it integer argument} 
$k$ with $0<k \leq n/b^*$.
 Here we also have to take care of the jump in $B_{x}=B_{k}$
 with respect to the index $x=k$, and conclude
\begin{equation}\label{sprung_fall2}
	\begin{split}
   & \lim \limits_{x \,\downarrow \,k}	
     \left(x \,B_{x}^-\left(\frac{n/x-b^*}{b}\right)
		\right)-
		    \lim \limits_{x \,\uparrow \, k}	
     \left(x \,B_{x}^-\left(\frac{n/x-b^*}{b}\right)
		\right)\\
		&=\lim \limits_{\varepsilon \,\downarrow \,0}	
		\left[
		\sum \limits_{j \leq k}\beta^-\left( j \cdot\frac{\frac{n}{k+\varepsilon}-b^*}{b}\right)
		-
		\sum \limits_{j < k}\beta^-\left(j \cdot \frac{\frac{n}{k-\varepsilon}-b^*}{b}\right)
		\right]\\
		&=\sum \limits_{j \leq k}\beta^- \left( j \cdot \frac{n/k-b^*}{b}\right)
		-\sum \limits_{j \leq k}\beta^+ \left(j \cdot \frac{n/k-b^*}{b}\right)+
		\beta \left(\frac{n-kb^*}{b} \right)\\
		&=\beta \left(\frac{n-kb^*}{b} \right)-
		k\left(B_{k}^{+}-B_{k}^{-} \right)\left(\frac{n/k-b^*}{b} \right)\,.
		\end{split}
  \end{equation}

Using \eqref{rntilde}, \eqref{sprung_fall2} we obtain
\begin{equation*}
	\begin{split}
     \left(R_n^{+}-R_n^{-}\right)\left(\frac{a}{b},k \right) 
		&=-\frac{b}{n}\,\beta \left(\frac{n-kb^*}{b}\right)\\
     &	-b\, ( B_{n}^{+}-
      B_{n}^{-}) \left( \frac{a}{b}+\frac{k}{bn}\right)\\	
      &+\frac{b}{n}k\, ( B_{k}^{+}-
      B_{k}^{-}) \left(\frac{n/k-b^*}{b}\right)\,.\\
  \end{split}
	\end{equation*}
Due to \eqref{phin_sprung} and Theorem \ref{farey_thm2} the second and third terms 
on the right-hand side cancel each other.

We conclude that $R_n$ is a step function with respect to $x$ for a given fraction 
$a/b$ which has jumps of height
\begin{equation*}
R_n^+\left(\frac{a}{b},k \right) -R_n^-\left(\frac{a}{b},k \right) 
=-\frac{b}{n}\,\beta \left(\frac{n-kb^*}{b}\right)
\end{equation*}
only at integer numbers $k$ with $0<k \leq n/b^*$. 
To complete the proof of the theorem we only have to note that
$\begin{displaystyle}
\lim \limits_{\varepsilon \, \downarrow \, 0} R_n(a/b,\varepsilon)=0\,.
\end{displaystyle}$
\end{proof}

Franel \cite{Fra1924} and Landau \cite{Lan1924} made use of the identity
\begin{equation}\label{landau}
\int \limits_0^1\beta(mx)\beta(nx)dx = \frac{(m,n)^2}{12 mn}\,,
\end{equation}
which is valid for all $m,n \in \setN$. A proof of this identity can be found 
in \cite[page 203]{Lan1924} as well as in Edward's textbook
\cite[Section 12.2]{Ed2001}. We need it for the following

\begin{thm}\label{Bx_L2}
For $x \to \infty$ we have with the $L_2(0,1)$-Norm $\|\cdot\|_2$
$$
\|B_{x}\|_2=\mathcal{O}\left(\frac{1}{\sqrt{x}}\right)\,.
$$
On the other hand we have a constant $C>0$ with
$$
\|B_{x}\|_2 \geq\frac{C}{\sqrt{x}} \quad \text{for~}x \geq 1\,.
$$
\end{thm}
\begin{proof}
We obtain from \eqref{landau}
\begin{equation*}
\begin{split}
\|B_{x}\|_2^2 &=\frac{1}{12x^2}\sum \limits_{m,n \leq x}\frac{(m,n)^2}{mn}
=\frac{1}{12x^2}\sum \limits_{d \leq x}\sum \limits_{\underset{(m,n)=d}{m,n\leq x\,:}}\frac{d^2}{mn}\\
&=\frac{1}{12x^2}\sum \limits_{d \leq x}\sum \limits_{\underset{(j,k)=1}{j,k\leq x/d\,:}}\frac{1}{jk}
\leq \frac{1}{12x^2}\sum \limits_{d \leq x}\sum \limits_{j,k\leq x/d}\frac{1}{jk}\\
&\leq \frac{1}{12x^2}\sum \limits_{d \leq x}\left(\log(x/d)+2\right)^2
=\mathcal{O}\left(\frac{1}{x}\right) \quad\text{for}~x \to \infty\\
\end{split}
\end{equation*}
with Euler's summation formula, regarding that
$$
\int \limits_{1}^{x}\left(\log(x/t)+2 \right)^2\,dt=10(x-1)-6\log(x)-\log(x)^2=\mathcal{O}\left(x \right)\,,
\quad x \geq 1\,.
$$
To complete the proof we note that
\begin{equation*}
\|B_{x}\|_2^2 =\frac{1}{12x^2}\sum \limits_{d \leq x}
\sum \limits_{\underset{(j,k)=1}{j,k\leq x/d\,:}}\frac{1}{jk}\geq 
\frac{1}{12x^2}\sum \limits_{d \leq x}1=\frac{\lfloor x \rfloor}{12x^2}\,.
\end{equation*}
\end{proof}

The next two theorems employ the elementary theory of continued fractions.
We will use them to derive estimates for $B_n(t)$ 
with $t$ in certain subsets $\mathcal{M}_n , \tilde{\mathcal{M}}_n \subset (0,1)$ and
$\lim \limits_{n \to \infty} |\mathcal{M}_n|=
\lim \limits_{n \to \infty}|\tilde{\mathcal{M}}_n| =1$.

First we recall some basic facts and notations about continued fractions.
For $\lambda_0 \in \setR$ and $\lambda_1,\lambda_2,\ldots,\lambda_m>0$
the finite continued fraction $\langle \lambda_0,\lambda_1,\ldots,\lambda_m \rangle$
is defined recursively by $\langle \lambda_0\rangle=\lambda_0$,
$\langle \lambda_0,\lambda_1\rangle=\lambda_0+1/\lambda_1$ and
\begin{equation*}
\langle \lambda_0,\lambda_1,\ldots,\lambda_m \rangle=
\langle \lambda_0,\ldots,\lambda_{m-2},\lambda_{m-1}+1/\lambda_m \rangle\,,
\quad m\geq 2\,.
\end{equation*}
Moreover, if $\lambda_j\geq 1$ is given for all $j \in \setN$, then the limit
\begin{equation*}
\lim \limits_{m \to \infty}\langle \lambda_0,\lambda_1,\ldots,\lambda_m \rangle=
\langle \lambda_0,\lambda_1,\lambda_2\ldots\rangle
\end{equation*}
exists and defines an infinite continued fraction. Especially for integer numbers
$\lambda_0 \in \setZ$ and  $\lambda_1,\lambda_2,\ldots \in \setN$ we obtain a unique
representation 
\begin{equation*}
t=\langle \lambda_0,\lambda_1,\lambda_2\ldots\rangle
\end{equation*}
for all $t \in \setR \setminus \setQ$ in terms of an infinite continued fraction.
For the determination of the coefficients $\lambda_j$ we need the following

\begin{defin}\label{vtdef}
For given  $t \in \setR \setminus \setQ$ we define a sequence of irrational numbers by 
\begin{equation*}
\vartheta_0=t\,, \quad \vartheta_j=\frac{1}{\vartheta_{j-1}-\lfloor \vartheta_{j-1}\rfloor}>1\,, \quad j \in \setN\,.
\end{equation*}
We may also write $\vartheta_j =\vartheta_j(t)$ in order to
indicate that the quantities $\vartheta_j$ depend on the fixed number $t$.
\end{defin}
We have
\begin{equation}\label{kette2}
\lambda_0= \lfloor t \rfloor\,,~\lambda_j = \lfloor \vartheta_j(t) \rfloor \quad \text{and~} 
t=\langle \lambda_0,\ldots,\lambda_{j-1},\vartheta_j(t)\rangle \quad  \text{for ~ all ~} j \in \setN\,.
\end{equation}

The following theorem is due to A. Ostrowski. 
It allows a very efficient calculation of the values $B_n(t)$
in terms of the continued fraction expansion of $t$. 

\begin{thm}\label{ostrowski}Ostrowski \cite[Equation (2), p. 80]{Ost1922}\\
Put $S(n,t)=\sum \limits_{k \leq n}\beta(kt)=nB_n(t)$ for $n \in \setN_0$ and $t \in \setR$.
Given are the continued fraction expansion $t=\langle\lambda_0,\lambda_1,\lambda_2,\ldots\rangle$
of any fixed $t \in \setR \setminus \setQ$ and $n \in \setN$. Then there is exactly
one index $j_* \in \setN$ with $b_{j_*} \leq n < b_{j_*+1}$, where
$\begin{displaystyle}
a_k/b_k = \langle \lambda_0,\ldots,\lambda_{k-1} \rangle
\end{displaystyle}$
are reduced fractions $a_k/b_k$ and $k,b_k \in \setN$. Put
$$
n'=n-b_{j_*}\left \lfloor \frac{n}{b_{j_*}}\right \rfloor\,.
$$
Then we have
\begin{equation}\label{Sost1}
 S(n,t)=S(n',t)+\frac{(-1)^{j_*}}{2}\left \lfloor \frac{n}{b_{j_*}}\right \rfloor
\left(1-\rho_{j_*} (n+n'+1) \right)
\end{equation}
with $\rho_{j_*}  = |b_{j_*} t-a_{j_*} |$ and
\begin{equation}\label{Sost2}
\left \lfloor \frac{n}{b_{j_*}}\right \rfloor \leq \lambda_{j_*}\,,\quad
0 < \left|1-\rho_{j_*}(n+n'+1) \right| <1\,.
\end{equation}
\end{thm}

Following Ostrowski's strategy we note two important conclusions.
We fix any number $t = \langle \lambda_0,\lambda_1,\lambda_2,\ldots \rangle \in \setR \setminus \setQ$ 
and apply Ostrowski's Theorem \ref{ostrowski} successively, starting with the calculation
of $S(n,t)$ and $\begin{displaystyle} |S(n,t)|\leq |S(n',t)| +\lambda_{j_*}/2\end{displaystyle}$\,.
If $n'=0$, then $S(n',t)=0$, and we are done. Otherwise we replace $n$ by the reduced number $n'$
with $0<n'<b_{j_*}$ and apply Ostrowski's Theorem again, and so on. For the final calculation
of $S(n,t)$ we need at most ${j_*}$ applications of the recursion formula and conclude from 
\eqref{Sost1}, \eqref{Sost2} that
\begin{equation}\label{Snfinal}
 n|B_n(t)| = |S(n,t)| \leq \frac12 \, \sum \limits_{k=1}^{j_*} \lambda_k \,.
\end{equation}
From $b_0=0$, $b_1=1$ and $b_{j+1}=b_{j-1}+\lambda_j b_j$ for $j \in \setN$
we obtain $b_{j+1}\geq 2b_{j-1}$, and hence for all $j \geq 3$ that
$\begin{displaystyle}
b_j \geq 2^{\frac{j-1}{2}}\,.
\end{displaystyle}$
Since $n \geq b_{j_*}$, we obtain without restrictions on $j_*$ for $n \geq 3$ that
$\begin{displaystyle}
n \geq 2^{\frac{j_*-1}{2}}
\end{displaystyle}$ and 
\begin{equation}\label{jstarabschaetz}
\begin{split}
j_* \leq 1+\frac{2}{\log 2 }\log n  \leq \left(1 + \frac{2}{2/3}\right)\log n = 4 \log n 
\,, \quad n \geq 3\,.
\end{split}
\end{equation}

We will see that \eqref{Snfinal} and \eqref{jstarabschaetz} have important conclusions.
An immediate consequence is Ostrowski's estimate \eqref{ostopt}
for irrational numbers $t$ with bounded partial quotients,
but first shed new light on these estimates by using Theorem \ref{Bx_thm}(b)
instead of Theorem \ref{ostrowski}. We put $J_*=(0,1)\setminus \setQ$ and fix 
any $t \in J_*$ and $n \in \setN$. The sequence
\begin{equation}\label{tsequence}
t_0 =t\,, \quad t_{j}=\frac{1}{t_{j-1}} - \left \lfloor\frac{1}{t_{j-1}} \right \rfloor
\end{equation}
with $j \in \setN$ is infinite, whereas the corresponding sequence of non-negative integer numbers
\begin{equation*}
n_0 =n\,, \quad n_{j} = \lfloor t_{j-1} n_{j-1} \rfloor
\end{equation*}
is strictly decreasing and terminates if $n_{j}=0$. 
Therefore $n_{j'}=0$ for some index $j'\in \setN$.
We assume $1 \leq j<j'$ and distinguish the two cases $0<t_{j-1}<1/2$ and $1/2<t_{j-1}<1$.
In the first case we have $n_{j+1}<n_{j} < n_{j-1}/2$,
and in the second case again
$$n_{j+1}=
\lfloor t_{j}\lfloor t_{j-1} n_{j-1}\rfloor\rfloor
< t_{j} t_{j-1} n_{j-1} = (1-t_{j-1})n_{j-1} < n_{j-1}/2\,.
$$
If $j'$ is odd, then
$$
n=n_0 \geq 2^{\frac{j'-1}{2}}n_{j'-1} \geq 2^{\frac{j'-1}{2}}\,,
$$
otherwise
$$
n=n_0 \geq n_1 \geq 2^{\frac{j'-2}{2}}n_{j'-1} \geq 2^{\frac{j'-2}{2}}\,,
$$
and $n \geq 2^{\frac{j'-2}{2}}$ in both cases. Therefore
\begin{equation}\label{jstrich}
j' \leq 2+ \frac{2}{\log 2} \log n \leq \left(1+ \frac{2}{\log 2}\right) \log n
\leq 4 \log n\,,\quad n \geq 8\,.
\end{equation}
Estimate \eqref{jstrich} bears a strong resemblance with \eqref{jstarabschaetz}. 
Now it follows from Theorem \ref{Bx_thm}(b) that
\begin{equation}\label{Bsequence}
B_n(t) = \sum \limits_{j=0}^{j'-1}(-1)^j
\left(
\frac{n_j}{n} {\tilde \eta}(t_j n_j)
+ \frac{t_j n_j-\lfloor t_j n_j \rfloor}{2n}
\right)\,.
\end{equation}
For the sequence in \eqref{tsequence} we have $\vartheta_{j+1}t_{j}=1$ for all $j \in \setN_0$,
and we obtain from the definition \eqref{etalimit} of ${\tilde \eta}$ that
\begin{equation*}
\frac{n_j}{n} {\tilde \eta}(t_j n_j)+ \frac{t_j n_j-\lfloor t_j n_j \rfloor}{2n}
=-\frac{t_j n_j-\lfloor t_j n_j \rfloor}{2n}
\left(
\vartheta_{j+1}\left(1-\left(t_j n_j-\lfloor t_j n_j \rfloor\right)\right)-1
\right)\,.
\end{equation*}
Here $\vartheta_{j+1}>1$ implies $|\vartheta_{j+1}\left(1-\left(t_j n_j-\lfloor t_j n_j \rfloor\right)\right)-1| \leq \max(1,\vartheta_{j+1}-1)$.
We see from \eqref{Bsequence} with Definition \ref{vtdef} and \eqref{kette2} that
\begin{equation}\label{Bestimate_modified}
\left | B_n(t)\right | \leq \sum \limits_{j=0}^{j'-1}
\frac{\max(1,\vartheta_{j+1}-1)}{2n}
\leq \frac{1}{2n}\sum \limits_{k=1}^{j'}\lambda_k\,.
\end{equation}
The calculations of $B_n(t)$ with Ostrowski's Theorem \ref{ostrowski} on one hand 
and with \eqref{Bsequence} on the other hand are similar but different.
Especially $j_*$ in Theorem \ref{ostrowski} and $j'$ used in \eqref{Bsequence}
are different in general. If we use \eqref{jstarabschaetz}
and \eqref{jstrich} then estimates \eqref{Bestimate_modified} and
\eqref{Snfinal} both give the same result. 
Hence Theorem \ref{Bx_thm}(b) may be used as well instead of Ostrowski's Theorem
for an efficient calculation and estimation of $B_n(t)$ and $S(n,t)$.
This is a surprising analogy.

\begin{thm}\label{kettenabschaetz}
Given are integer numbers $\alpha_1,\ldots,\alpha_m \in \setN$.
We put $J_*=(0,1)\setminus \setQ$.
Using Definition \ref{vtdef} 
with the functions $\vartheta_j$ depending on $t \in J_*$
we obtain for the measure $| \mathcal{M}|$ of the set
$$
\mathcal{M}=\left\{t \in J_*\,:\,\vartheta_j(t) < \alpha_j \quad\text{for~all~} j=1,\ldots,m\right\}
$$
the estimates
\begin{equation*}
\prod \limits_{j=1}^m \left(1-\frac{1}{\alpha_j} \right)^2
\leq | \mathcal{M}| \leq \prod \limits_{j=1}^m \left(1-\frac{1}{\alpha_j} \right)\,.
\end{equation*}
\end{thm}
\begin{proof}
The desired result is valid for $m=1$ with 
$\begin{displaystyle}
\mathcal{M}=\left\{t \in J_*\,:\,1/t < \alpha_1\right\}
\end{displaystyle}$
and 
$\begin{displaystyle}
|\mathcal{M}|=1-1/\alpha_1\,.
\end{displaystyle}$
Assume that the statement of the theorem is already true for a given $m \in \setN$.
We prescribe $\alpha_{m+1} \in \setN$ and will 
use induction to prove the statement for $m+1$. 

For all $j \in \setN$ and general given numbers $\lambda_0 \in \setR$ 
and $\lambda_1,\ldots,\lambda_{j-1}>0$
we put for $1\leq k<j$:
\begin{equation}\label{kette3}
\begin{split}
 a_0=1\,,~  a_1=\lambda_0\,,~ & a_{k+1}=a_{k-1}+\lambda_k a_k\,,\\
 b_0=0\,,~  b_1=1\,,~ & b_{k+1}=b_{k-1}+\lambda_k b_k\,.\\
\end{split}
\end{equation}
We have 
\begin{equation}\label{kette4}
\begin{split}
& \langle \lambda_0,\lambda_1, \ldots \lambda_{j-1},x \rangle -
\langle \lambda_0,\lambda_1, \ldots \lambda_{j-1},x' \rangle\\
&=
\frac{(-1)^j (x-x')}{(b_jx+b_{j-1})(b_jx'+b_{j-1})}\quad \text{for ~}x,x'>0\,.
\end{split}
\end{equation}
Especially for $\lambda_0=0$ and {\it integer numbers} $\lambda_1,\ldots,\lambda_j \in \setN$
we define the set $J(\lambda_1,\ldots,\lambda_j)$
consisting of all $t \in J_*$ between the two rational numbers
$\langle 0,\lambda_1, \ldots ,\lambda_{j-1},\lambda_{j}\rangle$ and
$\langle 0,\lambda_1, \ldots ,\lambda_{j-1},\lambda_{j}+1 \rangle$.

It follows from \eqref{kette3},\eqref{kette4} and all $j \in \setN$  that
\begin{equation}\label{kette5}
|J(\lambda_1,\ldots,\lambda_j)|=
\frac{1}{(b_j(\lambda_j+1)+b_{j-1})(b_j \lambda_j +b_{j-1})}\,.
\end{equation}
The sets $J(k)=(1/(k+1),1/k)\setminus \setQ$ with $k \in \setN$ 
form a partition of $J_*=(0,1)\setminus \setQ$. 
More general, it follows from Definition \ref{vtdef} and \eqref{kette2} 
for fixed numbers $\lambda_1,\ldots,\lambda_j \in \setN$ that
the pairwise disjoint sets $J(\lambda_1,\ldots,\lambda_j,k)$ with $k \in \setN$
form a partition of the set $J(\lambda_1,\ldots,\lambda_j)$.
We conclude by induction with respect to $j$ that the pairwise disjoint sets 
$J(\lambda_1,\ldots,\lambda_j)$ with $(\lambda_1,\ldots,\lambda_j) \in \setN^j$
also form a partition of $J_*$.

Now we put $j=m$ and distinguish two cases, $m$ odd and $m$ even, respectively.
In both cases, $m$ odd or $m$ even, the union
\begin{equation*}
\bigcup \limits_{k=1}^{\alpha_{m+1}-1}J(\lambda_1,\ldots,\lambda_m,k)
\end{equation*}
is the set of all numbers $t \in J_*$ with $\lfloor \vartheta_j(t) \rfloor = \lambda_j$
for $j=1,\ldots,m$ such that $\vartheta_{m+1}(t) < \alpha_{m+1}$. 
We define the set
\begin{equation*}
\mathcal{M}'=\left\{t \in J_*\,:\,\vartheta_j(t) < \alpha_j \quad\text{for~all~} j=1,\ldots,m+1\right\}
\end{equation*}
and conclude
\begin{equation}\label{mstrich}
|\mathcal{M}'| = \sum \limits_{\lambda_1=1}^{\alpha_{1}-1}
\sum \limits_{\lambda_2=1}^{\alpha_{2}-1}
\cdots
\sum \limits_{\lambda_m=1}^{\alpha_{m}-1}
\sum \limits_{k=1}^{\alpha_{m+1}-1}
|J(\lambda_1,\ldots,\lambda_m,k)|\,.
\end{equation}
It also follows from our induction hypothesis that
\begin{equation}\label{hypo}
\begin{split}
&\prod \limits_{j=1}^m \left(1-\frac{1}{\alpha_j} \right)^2\\
\leq&\sum \limits_{\lambda_1=1}^{\alpha_{1}-1}
\sum \limits_{\lambda_2=1}^{\alpha_{2}-1}
\cdots
\sum \limits_{\lambda_m=1}^{\alpha_{m}-1}
|J(\lambda_1,\ldots,\lambda_m)| \leq \prod \limits_{j=1}^m \left(1-\frac{1}{\alpha_j} \right)\,.\\
\end{split}
\end{equation}
We evaluate the inner sum in \eqref{mstrich}, and obtain for odd values of $m$
the telescopic sum
\begin{equation*}
\begin{split}
&\sum \limits_{k=1}^{\alpha_{m+1}-1}
|J(\lambda_1,\ldots,\lambda_m,k)|\\
&=
\sum \limits_{k=1}^{\alpha_{m+1}-1}
\left(
\langle 0,\lambda_1,\ldots,\lambda_m,k+1\rangle -\langle 0,\lambda_1,\ldots,\lambda_m,k\rangle
\right)\\
&=\langle 0,\lambda_1,\ldots,\lambda_m,\alpha_{m+1}\rangle -\langle 0,\lambda_1,\ldots,\lambda_m,1\rangle\\
&=\langle 0,\lambda_1,\ldots,\lambda_{m-1},\lambda_m+1/\alpha_{m+1}\rangle -
\langle 0,\lambda_1,\ldots,\lambda_{m-1},\lambda_m+1\rangle\,.\\
\end{split}
\end{equation*}
Apart from a minus sign on the right hand side we get the same result for even values of $m$,
and hence from \eqref{kette4} with $j=m$ in both cases 
\begin{equation}\label{tele}
\begin{split}
&\sum \limits_{k=1}^{\alpha_{m+1}-1}
|J(\lambda_1,\ldots,\lambda_m,k)|\\
&=
\frac{1-\frac{1}{\alpha_{m+1}}}{(b_m(\lambda_m+1)+b_{m-1})
(b_m(\lambda_m+\frac{1}{\alpha_{m+1}})+b_{m-1})}\,.\\
\end{split}
\end{equation}
Using $\lambda_m\geq 1$ we have
\begin{equation*}
\frac{\left(1-\frac{1}{\alpha_{m+1}}\right)^2}{b_m \lambda_m+b_{m-1}}
\leq
\frac{1-\frac{1}{\alpha_{m+1}}}{
b_m(\lambda_m+\frac{1}{\alpha_{m+1}})+b_{m-1}}
\leq
\frac{1-\frac{1}{\alpha_{m+1}}}{b_m \lambda_m+b_{m-1}}\,,
\end{equation*}
and obtain from \eqref{tele} and \eqref{kette5} with $j=m$ that
\begin{equation}\label{innerschaetz}
\begin{split}
&|J(\lambda_1,\ldots,\lambda_m)| \left(1-\frac{1}{\alpha_{m+1}}\right)^2\\
&\leq \sum \limits_{k=1}^{\alpha_{m+1}-1}|J(\lambda_1,\ldots,\lambda_m,k)|
\leq |J(\lambda_1,\ldots,\lambda_m)| \left(1-\frac{1}{\alpha_{m+1}}\right)\,.\\
\end{split}
\end{equation}
The theorem follows from \eqref{mstrich}, \eqref{hypo} and \eqref{innerschaetz}.
\end{proof}
\begin{rem}
Since $\alpha_j \in \setN$ for $j \leq m$, the conditions $\vartheta_j(t) < \alpha_j$
in the definition of the set $\mathcal{M}$ may likewise be replaced
by the equivalent conditions $\lambda_j \leq \alpha_j-1$, 
where $\lambda_j$ are the coefficients in the continued fraction expansion of $t$,
see Definition \ref{vtdef} and \eqref{kette2}\,.
\end{rem}


\section{Dirichlet series related to Farey sequences} \label{dirichlet_section}
We define the sawtooth function $\beta_0:\setR \to \setR$ by
\begin{equation*}
\beta_{0}(x)=\begin{cases}
x-\lfloor x \rfloor-\frac12
 &\ \text{for}\ x \in \setR \setminus{\setZ}\, ,\\
0 &\ \text{for}\ x \in \setZ\,.\\
\end{cases}
\end{equation*}

With $x>0$ the 1-periodic
function $B_{x,0}: \setR \to \setR$ is 
the arithmetic mean of $B_{x}^-$, $B_{x}^+=B_x$, see \eqref{Bxdef}, hence
\begin{equation}\label{Bx0def}
B_{x,0}\left(t \right)= \frac{1}{x}\sum \limits_{k \leq x} \beta_{0}(kt)
=\frac12\left(B_{x}^-(t) + B_{x}^+(t) \right)\,.
\end{equation}

\begin{lem}\label{B0_rational}
For (relatively prime) numbers $a \in \setZ$ and $b \in \setN$ we have 
$$
|B_{x,0}(a/b)| \leq \frac{b}{x}
$$
for all $x>0$.
\end{lem}
\begin{proof}
Without loss of generality we may assume that $a \in \setZ$ and $b \in \setN$
are relatively prime. Then Lemma 2.1 in \cite{Ku2} states that
\begin{equation}\label{zeromean}
\sum \limits_{m=1}^{b} \left\lfloor \frac{a m}{b} \right \rfloor
=a\frac{b+1}{2}-\frac{b-1}{2}\,.
\end{equation}
We can also assume that $b\geq2$, since $B_{x,0}(0)=0$. 
For $m \in \setN$ we define the $b$-periodic sequence
$$ t_{a/b}(m)=1+2b\beta\left(\frac{am}{b}\right)=
2am-2b\left\lfloor\frac{am}{b}\right\rfloor-b+1\,.$$
Due to \eqref{zeromean} this sequence
has mean value zero over one period, i.e.
\begin{equation*}
 \sum \limits_{m \leq \, b} t_{a/b}(m)=0\,.
\end{equation*}

We follow \cite[Section 2]{Ku2}, regard that $|\beta(t)| \leq \frac12$ 
for $t \in \setR$ and obtain
\begin{equation*}
\begin{split}
&\max \limits_{k \in \setN} 
\left|
\sum \limits_{m \leq k} 
t_{a/b}(m)
\right|\\
&=\max \limits_{1 \leq k \leq b} 
\left|
\sum \limits_{m \leq k} 
\left(2b \beta\left(\frac{am}{b} \right)+1 \right)
\right|\\
& \leq \max \limits_{1 \leq k \leq b} 
\sum \limits_{m \leq k} (b+1) = b(b+1) \,.\\
\end{split}
\end{equation*}
We conclude for $x>0$ that 
\begin{equation}\label{tabschaetz}
\left|\sum \limits_{m \leq x} t_{a/b}(m)\right| \leq b(b+1)\,.
\end{equation}
Next we use \eqref{Bxjump} and obtain
\begin{equation*}
\begin{split}
B_{x,0}(a/b)&=
B_x(a/b)+\frac{1}{2x}\left\lfloor\frac{x}{b}\right\rfloor\\
&=\frac{1}{2bx}\sum \limits_{m \leq x}\left(t_{a/b}(m)-1 \right)+\frac{1}{2x}\left\lfloor\frac{x}{b}\right\rfloor\\
&=\frac{1}{2bx}\sum \limits_{m \leq x} t_{a/b}(m)+
\frac{1}{2x}\left(\left\lfloor\frac{\lfloor x \rfloor}{b}\right\rfloor-\frac{\lfloor x \rfloor}{b}\right)\,.\\
\end{split}
\end{equation*}
Hence we see from \eqref{tabschaetz} with $b \geq 2$ that
\begin{equation*}
\left| B_{x,0}(a/b) \right| \leq \frac{b+1}{2x}+\frac{1}{2x} \leq \frac{b}{x}\,.
\end{equation*}
\end{proof}

Using Theorem \ref{Bx_thm}(a), Lemma \ref{B0_rational}, \eqref{Bx0def}, \eqref{Bxjump}
and for $x \in \setR$ the symmetry relationship
\begin{equation*}
B_n\left(\frac{a}{b}-\frac{x}{bn}\right)=-B_{n}^-\left(\frac{b-a}{b}+\frac{x}{bn}\right)\,,
\end{equation*}
we obtain the following result, which has the counterparts \cite[Theorem 3.2]{Ku2}
and \cite[Theorem 2.2]{Ku4} in the theory of Farey fractions:

\begin{thm}\label{rescaled_limit}
Assume that $a/b \in \mathcal{F}^{ext}_n$ and put
$$
\tilde{\eta}_{a,b}(n,x)=b\,B_n\left(\frac{a}{b}+\frac{x}{bn}\right) \,,
\qquad x \in \setR\,.
$$
Then for $n \to \infty$ the sequence of functions $\tilde{\eta}_{a,b}(n,\cdot)$ converges uniformly 
on each interval $[-x_*,x_*]$, $x_*>0$ fixed, to the limit function $\tilde{\eta}$
in \eqref{etalimit}.
\end{thm}

For the following two results we apply Theorem \ref{ostrowski} 
and recall \eqref{Snfinal}, \eqref{jstarabschaetz}.
Due to Theorem \ref{rescaled_limit} the functions $B_n$ cannot converge 
uniformly to zero on any given interval. Instead we have the following

\begin{thm}\label{B0_mass}
Let $\Theta : [1,\infty) \to [1,\infty)$ be monotonically increasing with
$\begin{displaystyle} \lim \limits_{n \to \infty} \Theta(n)=\infty\,.\end{displaystyle}$

We fix $n \in \setN$, $\varepsilon >0$, put $m=\lfloor 4 \log n\rfloor$,
use Definition \ref{vtdef}, recall $J_*=(0,1) \setminus \setQ$ and define 
$$\mathcal{M}_n=
\left\{t \in J_*\,:\,\vartheta_j(t) <1+\lfloor\Theta(n)\log n\rfloor \quad 
\mbox{for~all ~ } j=1,\ldots,m \right\}\,.
$$ 
Then $\begin{displaystyle}  \lim \limits_{n \to \infty}|\mathcal{M}_n|=1\end{displaystyle}$
and 
\begin{equation}
\label{betteresti}
\left| B_{n,0}(t) \right|=\left| B_{n}(t) \right| \leq 2 \, \frac{\log^2 n}{n}\,\Theta(n) 
\end{equation}
for all $n \geq 3$ and all $t \in \mathcal{M}_n$.
\end{thm}
\begin{proof}
We apply Ostrowski's Theorem on any number $t \in \mathcal{M}_n$ with continued fraction expansion
$t = \langle 0,\lambda_1,\lambda_2,\ldots \rangle$ and 
obtain $j_* \leq m$ from \eqref{jstarabschaetz}, since $j_*$ is an integer number. 
From $j_* \leq m$ and $t \in \mathcal{M}_n$
we conclude that $\lambda_k \leq \Theta(n) \log n$ for $k=1,\ldots,j_*$,
and the desired inequality follows with \eqref{Snfinal}\,.
The first statement follows from Theorem \ref{kettenabschaetz} via
\begin{equation*}
\begin{split}
|\mathcal{M}_n| &\geq \left(1-\frac{1}{ 1+\lfloor\Theta(n)\log n\rfloor} \right)^{2m}\\
&\geq \left(1-\frac{1}{\Theta(n)\log n} \right)^{2m}
\geq \left(1-\frac{1}{\Theta(n)\log n} \right)^{8\log n}\,,\\
\end{split}
\end{equation*}
since the right-hand side tends to $1$ for $n \to \infty$.
\end{proof}

\begin{rem} The sets $\mathcal{M}_n$ in the previous theorem are chosen in such a way that 
the large values $B_n(t)$ from the peaks of the rescaled limit function
around  the rational numbers with small denominators
predicted by Theorem \ref{rescaled_limit} can only occur in the small complements 
$J_* \setminus \mathcal{M}_n$ of these sets.
However, the quality of the estimates of the values $B_n(t)$ 
on the sets $\mathcal{M}_n$ depends on the
different choices of the growing function $\Theta$. For example,
$\Theta(n)=1+\log\left(1+\log n\right)$ gives a much smaller bound
than $\Theta(n)=16 \sqrt{n}/(4+\log n)^2$, whereas the latter choice leads to 
a much smaller value of $|J_* \setminus \mathcal{M}_n|=1-|\mathcal{M}_n|$\,.
\end{rem}

\begin{thm}\label{B0_mass_almost_everywhere}
Let $\Theta : [1,\infty) \to [1,\infty)$ be monotonically increasing with
$\begin{displaystyle} \lim \limits_{n \to \infty} \Theta(n)=\infty\,.\end{displaystyle}$
We fix $n \in \setN$, $\varepsilon >0$, use Definition \ref{vtdef}, recall $J_*=(0,1) \setminus \setQ$ and put
$$\tilde{\mathcal{M}_n}=
\left\{t \in J_*\,:\,\vartheta_j(t) < 1+\lfloor\Theta(n)j^{1+\varepsilon}\rfloor \quad\text{for~all~} j \in \setN \right\}\,.
$$ 
Then $|\tilde{\mathcal{M}}|=1$ for 
$\tilde{\mathcal{M}}= \bigcup \limits_{n=1}^{\infty}\tilde{\mathcal{M}_n}$, 
and for all $t \in \tilde{\mathcal{M}}$ there exists an index $n_0=n_0(t,\varepsilon)$ with
$$\left| B_{n,0}(t) \right|=\left| B_{n}(t) \right| 
\leq \frac{(4\log n)^{2+\varepsilon}}{2n}\,\Theta(n)
\quad \mbox{for ~ all ~} n \geq n_0\,.
$$ 
The complement $J_* \setminus \tilde{\mathcal{M}}$ is an uncountable null set
which is dense in the unit interval $(0,1)$.
\end{thm}
\begin{proof}
The function $\Theta$ is monotonically increasing, hence
$
\tilde{\mathcal{M}_1} \subseteq \tilde{\mathcal{M}_2}\subseteq \tilde{\mathcal{M}_3}\ldots\,,
$
and we have
\begin{equation}\label{kettenmass1}
|\tilde{\mathcal{M}}|=
\lim \limits_{n \to \infty} |\tilde{\mathcal{M}_n}|\,.
\end{equation}
For all $n,k \in \setN$ we define
$$\tilde{\mathcal{M}_{n,k}}=
\left\{t \in J_*\,:\,\vartheta_j(t) < 1+\lfloor\Theta(n)j^{1+\varepsilon}\rfloor 
\quad\text{for~all~} j =1,\ldots,k \right\}\,.
$$ 
Then
$\tilde{\mathcal{M}_n}= \bigcap \limits_{k=1}^{\infty}\tilde{\mathcal{M}_{n,k}}$ 
and 
\begin{equation}\label{kettenmass2}
|\tilde{\mathcal{M}_n}|=
\lim \limits_{k \to \infty} |\tilde{\mathcal{M}_{n,k}}|
\end{equation}
from
$
\tilde{\mathcal{M}_{n,1}} \supseteq \tilde{\mathcal{M}_{n,2}}\supseteq \tilde{\mathcal{M}_{n,3}}\ldots\,.
$
It follows from Theorem \ref{kettenabschaetz} for all $n,k \in \setN$ that
\begin{equation*}
|\tilde{\mathcal{M}_{n,k}}|
\geq \prod \limits_{j=1}^{k}\left(1-\frac{1}{\Theta(n)j^{1+\varepsilon}} \right)^2
\geq \prod \limits_{j=1}^{\infty}\left(1-\frac{1}{\Theta(n)j^{1+\varepsilon}} \right)^2\,.
\end{equation*}
The product on the right-hand side is independent of $k$ and converges to $1$
for $n \to \infty$, hence $|\tilde{\mathcal{M}}|=1$ from
\eqref{kettenmass1}, \eqref{kettenmass2}\,.
Each rational number in the interval $(0,1)$ is arbitrarily close to a member of
the complement $J_* \setminus \tilde{\mathcal{M}}$, and the complement contains
all $t=\langle 0,\lambda_1,\lambda_2,\lambda_3,\ldots \rangle$
for which $(\lambda_j)_{j \in \setN}$ increases faster then any polynomial.
We conclude that $J_* \setminus \tilde{\mathcal{M}}$
is an uncountable null set which is dense in the unit interval $(0,1)$.
Now we choose $t \in \tilde{\mathcal{M}}$ and obtain $n_0 \in \setN$
with $t \in \tilde{\mathcal{M}}_{n_0}$. Then $t \in \tilde{\mathcal{M}}_{n}$
for all $n \geq n_0$, and we may assume that $n_0 \geq 3$.
Note that $n_0$ may depend on $t$ as well as on $\varepsilon$.
We have $t = \langle 0,\lambda_1,\lambda_2,\lambda_3,\ldots \rangle$ and
$$
 \lambda_j \leq \Theta(n)j^{1+\varepsilon}
$$
for all $n \geq n_0$ and all $j \in \setN$. 
We finally obtain from \eqref{Snfinal}, \eqref{jstarabschaetz} that
$$
n|B_n(t)| = |S(n,t)| \leq \frac12 \, \sum \limits_{k=1}^{j_*} \lambda_k
\leq \frac{j_*}{2}\Theta(n)j_*^{1+\varepsilon}
\leq \frac12 \Theta(n) (4\log n)^{2+\varepsilon}\,, \quad n \geq n_0\,.
$$
\end{proof}

\begin{rem}
We replace $\varepsilon$ by $\varepsilon/2$,
choose $\Theta(n)=1+\log\left(1+\log n\right)$ in the previous theorem
and obtain the following result of Lang, see \cite{Lang1966}
and \cite[III,\S 1]{Lang1995} for more details:
For $\varepsilon > 0$ and almost all $t \in \setR$ we have 
\begin{equation*}
|S(n,t)| \leq\left(\log n \right)^{2+\varepsilon} \quad
\mbox{~for ~} n \geq n_0(t,\varepsilon)
\end{equation*}
with a constant $n_0(t,\varepsilon) \in \setN$.
Here the sum $S(n,t)$ is given by \eqref{Sxdef}\,.
This doesn't contradict Theorem \ref{Bx_L2}, because the pointwise estimates of $S(n,t)$ and
$B_n(t)$ in Theorem \ref{B0_mass_almost_everywhere} are only valid for sufficiently large values of $n \geq n_0(t,\varepsilon)$, depending on the choice of $t$ and $\varepsilon$.

We conclude from Theorem \ref{B0_mass} that the major contribution of $\|B_n\|_2$
comes from the small complement of $\mathcal{M}_n$. Indeed, the crucial point
in Theorem \ref{B0_mass} is that it holds for \textit{all} $n \geq 3$, but
not so much the fact that the upper bound in estimate \eqref{betteresti} is slightly better
than that in Theorem \ref{B0_mass_almost_everywhere}.
\end{rem}

For $k \in \setN$ and $x>0$ the 1-periodic
functions $q_{k,0},\Phi_{x,0} : \setR \to \setR$ 
corresponding to \eqref{familien} are 
defined as follows:
\begin{equation*}
q_{k,0}(t) =-\sum \limits_{d | k} \mu(d)\,
\beta_{0} \left(\frac{kt}{d}\right) ~\,,
\end{equation*}
\begin{equation*}
\Phi_{x,0}\left(t \right)= \frac{1}{x}\,\sum \limits_{k \leq x} q_{k,0}(t)
=- \frac{1}{x}\sum \limits_{j \leq x}
\sum \limits_{k \leq x/j}\mu(k)\,\beta_{0}\left(jt \right)\,.
\end{equation*}

In the half-plane $H = \{ s \in \setC\,:\,\Re(s)>1\}$
the parameter-dependent Dirichlet series $F_{\beta}, F_{q} : \setR \times H \to \setC$ are given by
\begin{equation*}
F_{\beta}(t,s) =\sum \limits_{k=1}^{\infty} \frac{\beta_{0}(kt)}{k^s} \,,\quad
F_{q}(t,s) =\sum \limits_{k=1}^{\infty} \frac{q_{k,0}(t)}{k^s} \,.
\end{equation*}
Now Theorem \ref{B0_mass_almost_everywhere} and \eqref{Hecke_series} immediately gives
\begin{thm}\label{F_thm}
For $t \in \setR$ and $\Re(s)>1$ we have with absolutely convergent series and integrals
\begin{itemize}
\item[(a)] 
\begin{equation*}
\begin{split}
\frac{1}{s}\,F_{\beta}(t,s)=\int \limits_{1}^{\infty} B_{x,0}(t)\frac{dx}{x^s}\,,\quad
\frac{1}{s}\,F_{q}(t,s)=\int \limits_{1}^{\infty} \Phi_{x,0}(t)\frac{dx}{x^s}\,.\\
\end{split}
\end{equation*}
\item[(b)] 
\begin{equation*}
F_{q}(t,s)=-\frac{1}{\zeta(s)}\,F_{\beta}(t,s)\,.
\end{equation*}
\end{itemize}
For almost all $t$ the function $F_{\beta}(t,\cdot)$ 
has an analytic continuation to the half-plane $\Re(s)>0$\,.
\end{thm}

\section{Appendix: Plots of the limit functions $h$ and $\tilde{\eta}$}\label{appendix}
\begin{center}
\begin{figure}[H]
\includegraphics[width=0.7\textwidth]{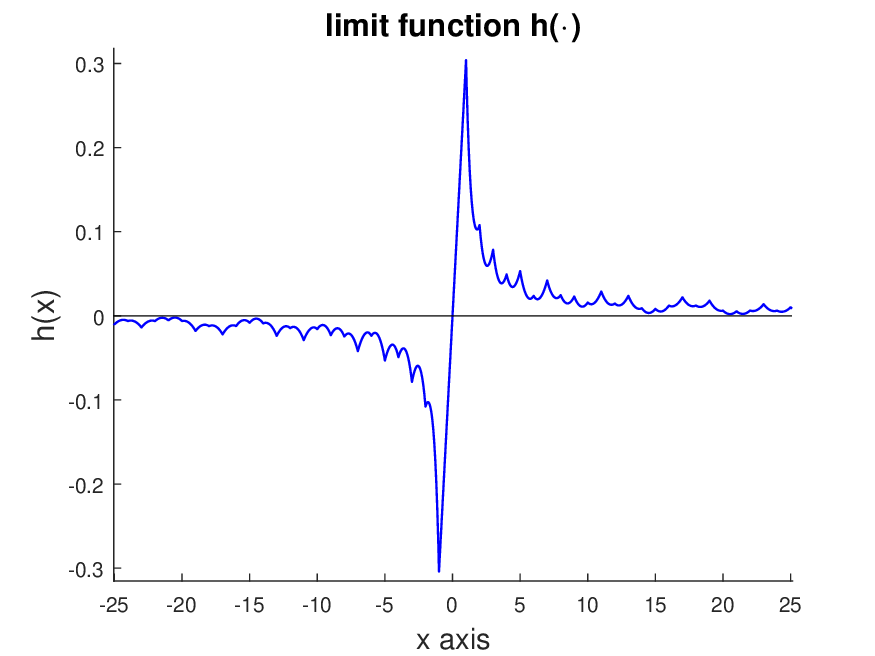}
\caption{Plot of $h(x)$ for $-25 \leq x \leq 25$\,. \label{h25eps}}
\end{figure}
\end{center}
\begin{center}
\begin{figure}[H]
\includegraphics[width=0.7\textwidth]{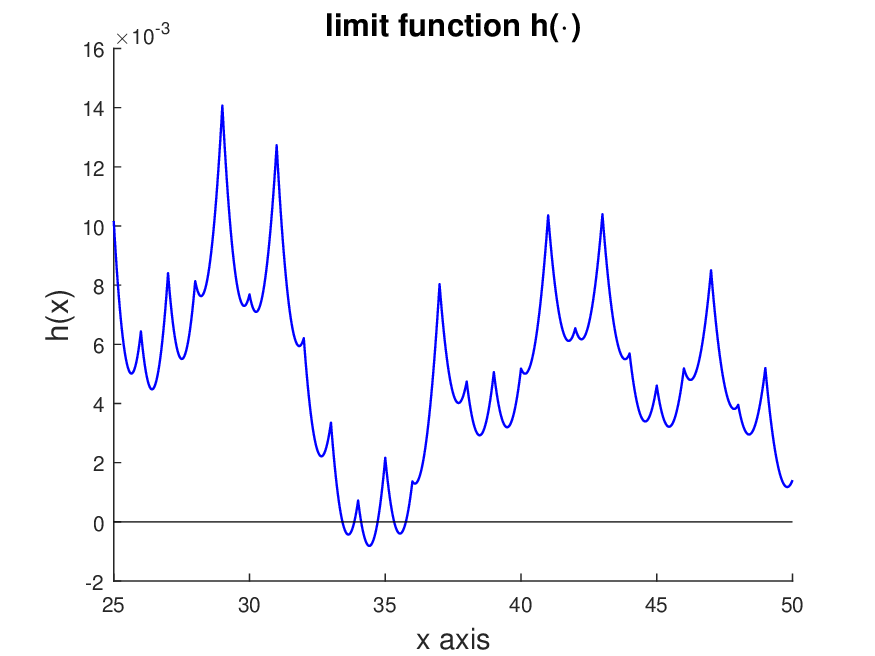}
\caption{Plot of $h(x)$ for $25 \leq x \leq 50$\,.  \label{h50eps}}
\end{figure}
\end{center}
\begin{center}
\begin{figure}[H]
\includegraphics[width=0.7\textwidth]{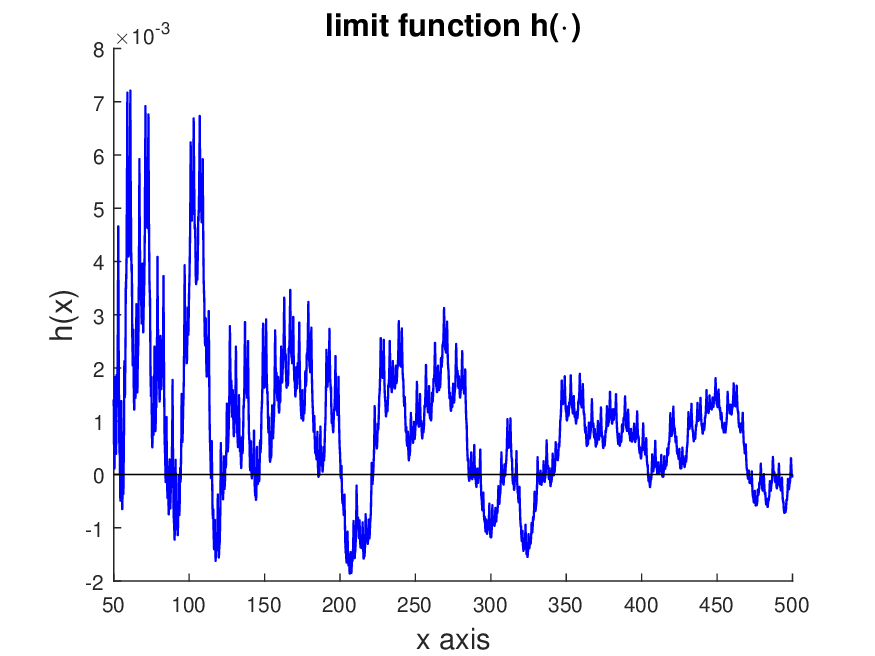}
\caption{Plot of $h(x)$ for $50 \leq x \leq 500$\,. \label{h500eps}}
\end{figure}
\end{center}
\begin{center}
\begin{figure}[H]
\includegraphics[width=0.7\textwidth]{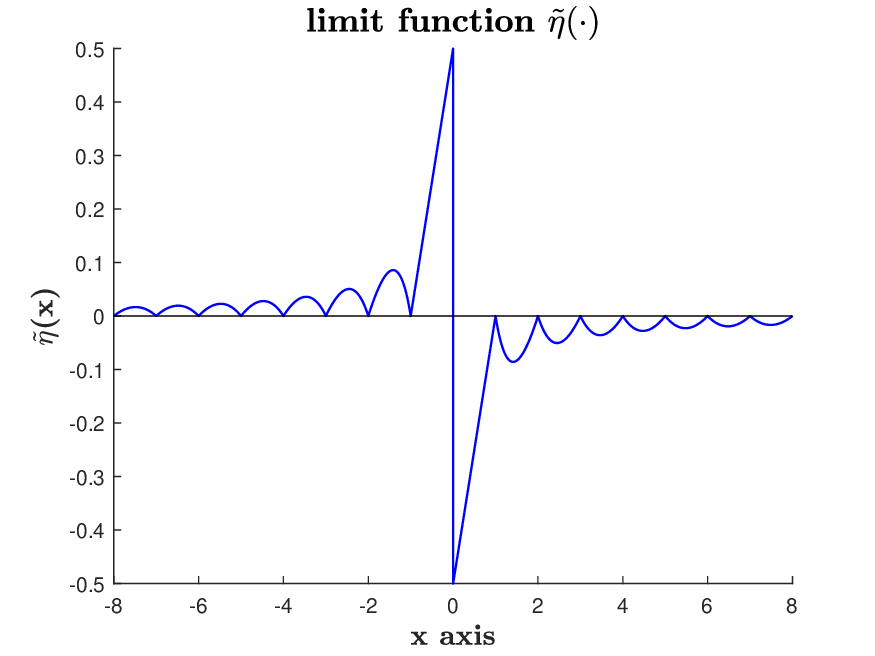}
\caption{Plot of $\tilde{\eta}(x)$ for $-8 \leq x \leq 8$\,. \label{eta8eps}}
\end{figure}
\end{center}


\end{document}